\documentclass[preprint,12pt]{elsarticle}

\usepackage{amssymb}
\usepackage{amsthm}

\newtheorem{corollary}{Corollary}

\newtheorem{proposition}{Proposition}

\newtheorem{definition}{Definition}

\renewcommand{\geq}{\geqslant} 
\renewcommand{\leq}{\leqslant} 
\renewcommand{\ge}{\geqslant} 
\renewcommand{\le}{\leqslant}

\newcommand{\glebdone}[1]{}

\newcommand{\wudone}[1]{}

\newcommand{\simondone}[1]{}

\usepackage{amsmath}
\usepackage{graphicx}
\usepackage[colorinlistoftodos]{todonotes}
\usepackage[colorlinks=true, allcolors=blue]{hyperref}
\usepackage[english]{babel}
\usepackage[T1]{fontenc}
\usepackage{lmodern}
\usepackage{amssymb}
\usepackage{enumitem}
\usepackage{amsmath}
\usepackage[ruled]{algorithm2e}

\begin{document}

\begin{frontmatter}



\title{Generating fuzzy measures from additive measures}


\author[inst1]{Jian-Zhang Wu \corref{cor1}}
\ead{jianzhang.wu@deakin.edu.au}
\cortext[cor1]{Corresponding author}
\author[inst1]{Gleb Beliakov}

\affiliation[inst1]{organization={School of Information Technology, Deakin University},
            addressline={Burwood}, 
            postcode={3125}, 
            country={Australia}}



\begin{abstract}
Fuzzy measures, also referred to as nonadditive measures, emerge from the foundational concept of additive measures by transforming additivity into monotonicity. In comparison to the expansive $2^n$ coefficients of fuzzy measures, additive measures encompass just $n$ coefficients, enabling them to efficiently provide initial viable solutions across various domains including normal, super/submodular, super/subadditive, (anti)buoyant, and other specialized fuzzy measures. To further enhance the effectiveness of these measures, techniques such as allowable range adjustments and random walks have been introduced, aiming to decrease the redundancy in linear extensions and bolster the random or uniform nature of the generated measures. In addition to innovating the ideas of random generation and adjustment strategies for multiple types of fuzzy measures, this paper also sheds light on the profound connection between additive and fuzzy measures.

\end{abstract}




\end{frontmatter}


\section{Introduction}

Fuzzy measures \cite{sugeno1974theory}, commonly known as nonadditive measures \cite{denneberg1994non}, emerge as a natural extension of the fundamental concept of additive measures by replacing the rigid notion of strict additivity with the more flexible principle of monotonicity with respect to set inclusion \cite{grabisch2016:setfunctionbook}. This transition not only imparts fuzzy measures with a distinct character but also empowers them to elegantly encapsulate the complexity of diverse interaction scenarios and flexibly capture the richness and subtleties of dependent multifaceted information \cite{Beliakov2007_book,beliakov2019discrete_book}.


Specialized structured fuzzy measures play a significant role in various fields such as decision-making \cite{grabisch1995fuzzy} and economics \cite{beliakov2022choquet}, as well as in the realm of optimization computing \cite{beliakov-wu-towards}. These measures are designed with specific patterns or properties that offer precise interpretations and insights in practical scenarios, or contribute to a reduction in computational complexity \cite{grabisch2008review}. 
For instance, super/submodular measures and super/subadditive measures signify that all elements or subsets have positive interactions with each other \cite{wuBeliakovNonadd,wuBeliakovNonmodu}. Antibuoyant measures establish a connection between the behavior of the Choquet integral and the Pigou-Dalton progressive transfers principle for societal equality \cite{beliakov2021choquet}.
$p$-symmetric measures are specifically designed to address complex indifference cases among referrers and criteria in decision-making and evaluations \cite{miranda2002p}.
$k$-intolerance and $k$-tolerant measures showcase the extreme cases of either fully supporting or vetoing any $k$ decision criteria or voters \cite{marichal2004tolerant}. On the other hand, $k$-maxitive and $k$-minitive measures correspond to the notions of possibility and necessity for $k$ and higher-order subsets \cite{beliakov2020kmaxtive}. 
$k$-interactive measures establish a straightforward structure for $k$ and higher order measure values exhibiting their partial maximum entropy \cite{beliakov2011learning}.
Furthermore, $k$-additive measures \cite{grabisch1997k}, $k$-nonadditive measures \cite{wuBeliakovNonadd}, and $k$-nonmodular measures \cite{wuBeliakovNonmodu} are designed to simplify calculations by disregarding interactions involving $k$ and higher-order elements or criteria, where the simplification is based on the M\"obius representation, nonadditivity index, and nonmodularity index, respectively.


Additive measures, being a distinctive type of fuzzy measure, possess a limited set of $n$ coefficients that align with the cardinality of the universal set containing $n$ elements. This attribute offers distinct computational benefits when generating individuals. Fortunately, additive measures also fall within, or can be easily manipulated to fit into, the category of specialized structured fuzzy measures mentioned earlier. For example, additive measures can be seen as super/submodular, super/subadditive measures and general cases of $k$-additive, $k$-nonadditive and $k$-nonmodular measures; uniform (additive and symmetric) measures are also (anti)buoyant measures; additive measure are easily changed into $k$-intolerance and $k$-tolerant measures, $k$-maxitive and $k$-minitive measure and $k$-interactive measures.
Consequently, additive measures can be employed as initial individuals in the random generation algorithms for these specific types of measures.


Given that additive measures encompass only a limited domain within the realm of fuzzy measures and their specific variations, they exhibit a deficiency in terms of diversity and randomness. This is particularly evident in their high repetition rate of linear extensions, particularly for smaller universal sets.

The subsequent three main techniques have been employed to enhance the diversity and randomness of normal fuzzy measures and all the associated specialized variations.
The first technique, called the allowable range, seeks to identify a suitable range or interval for adjusting the measure value of a selected subset while adhering to conditions such as monotonicity, nonmodularity, nonadditivity, or specific requirements inherent to specialized structured fuzzy measures. This process involves carefully considering neighboring subsets, supersets of subset $A$, as well as subsets linked to $A$ through M"obius representations, nonadditivity indices, or nonmodularity indices.
The second technique is the random walk approach, which focuses on making controlled adjustments to the position of a selected subset within the linear extension. This adjustment can involve swapping the position of the chosen subset either upwards or downwards with its adjacent subsets, all while ensuring that the given conditions or requirements are maintained. It's important to note that the random walk technique is particularly effective when combined with the allowable range technique.
The third technique involves direct adjustment strategies that specifically target particular scenarios, such as transforming an additive measure into a strictly supermodular or superadditive measure, and eliminating oscillations that arise from lower order subsets and affect higher order subsets in $k$-order fuzzy measures.
It's important to note that these techniques also contribute to our understanding of the structural and conceptual distinctions between fuzzy and additive measures.

This paper is structured as follows: following the introduction, Section 2 introduces the concepts of additive measures, fuzzy measures, and specialized measures. Section 3 discusses the adjustment technique of additive measures to normal fuzzy measures. In Section 4, we investigate the techniques for adjusting additive measures to supermodular and antibuoyant measures. Section 5 is devoted to generating superadditive measures from additive measures. In Section 6, we present methods for generating $p$-symmetric measures, $k$-tolerant measures, $k$-maxitive measures, $k$-additive measures, $k$-nonadditive measures, and $k$-nonmodular measures. Finally, Section 7 provides the concluding remarks.
	
\section{Preliminaries}

Let $N = \{ 1,2, \ldots ,n\}$, $n \ge 2$, be the discrete set of argument items or criteria, $\mathcal{P}(N)$ the power set of $N$, and $|S|$ the cardinality of a subset $S\subseteq N$. 
 \begin{definition} \cite{beliakov2019discrete_book} \label{def.add.measure}
 An additive measure on $N$ is a set function $\mu :\mathcal {P}(N) \to [0,1]$
		such that
		\begin{enumerate}
			\item [(i)]{$\mu (\emptyset ) = 0, \mu (N) = 1;$ (boundary condition)}
			\item [(ii)]{ $\mu(A \cup B)= \mu(A) + \mu(B)$, $\forall A, B \subseteq N$, $A \cap B = \emptyset$. (additivity condition)}
		\end{enumerate}
\end{definition}

 Fuzzy measures (also called nonadditive measures \cite{denneberg1994non}, capacities \cite{choquet1954}), extend the additivity condition of additive measures to the monotonicity condition, thereby achieving a more flexible representation ability \cite{grabisch1995fuzzy}.
 
\begin{definition} \cite{sugeno1974theory} \label{def. fuzzy measure}
A fuzzy measure on $N$ is a set function $\mu :\mathcal {P}(N) \to [0,1]$ such that
\begin{enumerate}
\item [(i)]{$\mu (\emptyset ) = 0, \mu (N) = 1;$ (boundary condition)}
\item [(ii)]{$\forall A,B \subseteq N$, $A \subseteq B$ implies $\mu (A) \le \mu (B)$. (monotonicity condition)}
\end{enumerate}
\end{definition}


The ordering of subsets in correspondence with fuzzy measure values serves as a valuable tool for depicting the uniformity or randomness of a group of fuzzy measures within the feasible domain.

\begin{definition} \cite{combarro201950}
    A linear extension of all subsets in $N$ is a linear or total order $\precsim$ on $\mathcal {P}(N)$ such that $A \subseteq B$ implies  $A \precsim B$, where $\precsim$ is reflexive, antisymmetric, and transitive, and any elements in $\mathcal {P}(N)$ are comparable with order  $\precsim$ .
\end{definition}

Nonadditivity and nonmodularity are two distinct interaction characteristics of fuzzy measures.

\begin{definition} \cite{wuBeliakovNonadd,beliakov2019discrete_book} \label{def.superadd}
A fuzzy measure $\mu$ on $N$ is called superadditive (subadditive) if 
     $\mu(A \cup B) \geq (\leq) \mu(A) + \mu(B)$, $\forall A, B \subseteq N$, $A \cap B = \emptyset$
\end{definition}

\begin{definition} \label{def. non modular fuzzy measure }
\cite{wuBeliakovNonmodu,grabisch2016:setfunctionbook} A fuzzy measure $\mu$ on $N$ is called supermodular 
 (submodular) if $\mu(A \cup B) + \mu(A \cap B) \geq (\leq) \mu(A) + \mu(B)$, $\forall A, B \subseteq N$, $A \cap B \neq A, B$. 
\end{definition}



The $p$-symmetric measure represents subsets from the viewpoint of indifference elements and maps them to $p$-dimensional vectors.

\begin{definition} \cite{miranda2002p}
Let $\mu$ be a fuzzy measure on $N$, a subset $A \subseteq N$ is a subset of indifference with respect to $\mu$ if $\forall B_1,B_2 \subseteq A$, $|B_1| = |B_2|$, then $\mu(C\cup B_1) = \mu (C \cup B_2),\forall C \subseteq N\setminus A$.
\end{definition}

\begin{definition} \cite{miranda2002p} \label{def.p.symmetric}
A capacity $\mu $ on $N$ is said to be  $p$-symmetric if the coarsest partition of $N$ into subsets of indifference contains exactly $p$ subsets $A_1,..., A_p$, where $A_i$ is a subset of indifference, $A_i \cap A_j =\emptyset$,
$\cup_{i=1}^p A_i=N$, $i, j = 1,..., p$, and a partition $\pi$ is coarser than another partition $\pi'$
 if all subsets of $\pi$  are union of
some subsets of $\pi'$. The partition $\{A_1,...,A_p\}$ is called the basis of $\mu$.
\end{definition}


In order to reduce variables as well as their monotonicity conditions involved in constructing fuzzy measure, the notion of $k$-order representative fuzzy measure is proposed \cite{wuBeliakov-k-representative}.  Some simple instances are  $k$-tolerant measure\cite{marichal2007k}, $k$-interactive measure \cite{beliakovWuLearning} and $k$-maxitive measure \cite{mesiar1999generalizations,mesiar2018k,wuBeliakovkminitive}.

\begin{definition} \cite{marichal2007k} \label{def. tole and intole}
	Let $k \in \{1,..., n\}$, a fuzzy measure $\mu$ on $N$ is said to be $k$-tolerant if $\mu(A) = 1$ for all $A \subseteq N$ such that $|A| \geq k$ and there exists a subset
	$ B \subseteq N$, with $|B| = k-1$, such that $\mu(B) \neq 1$. 
\end{definition}



\begin{definition}\label{def:kinter}
	A fuzzy measure  is called $k$-interactive if for some chosen $K\in [0,1], \mu(A)=K+\frac{|A|-k-1}{n-k-1}(1-K), \forall A \subseteq N, |A|\geq k+1.$
\end{definition}




\begin{definition}\cite{mesiar1999generalizations,mesiar2018k,wuBeliakovkminitive} \label{def-k-maxitive fuzzy measure}
	Let $k \in \{1,..., n\}$. A fuzzy measure $\mu$ is said to be $k$-maxitive if its $k+1$ and higher orders' fuzzy measure values are obtained by $\mu (A)=\vee_{B\subset A,|B|= k} \mu(B), |A|\geq k+1.$

\end{definition}

Some more complex  types of $k$-order representative measures include $k$-additive measure \cite{grabisch1997k}, $k$-nonadditive measure and $k$-nonmodular measure \cite{wuBeliakovNonadd,wuBeliakovNonmodu,wuBeliakov-k-representative}.

\begin{definition} \cite{grabisch1997k}
Let $k \in \{1,..., n\}$, a fuzzy measure $\mu$ on $N$ is said to be $k$-additive if its Möbius representation satisfies $m_\mu(A) = 0$ for all $A \subseteq N$ such that $|A| > k$ and there exists at least one subset $A$ of $k$ elements such that $m_\mu(A) \neq 0$, where the M\"obius representation of subset $A \subseteq N$ of $\mu $ is defined as
${m_\mu }(A) = \sum_{C \subseteq A} {{{( - 1)}^{|A\backslash C|}}\mu (C)}.$
\end{definition}

\begin{definition} \label{def. k nonadd} \cite{wuBeliakovNonadd,wuBeliakov-k-representative}
Let $k \in \{1,..., n\}$, a fuzzy measure $\mu$ on $N$ is said to be $k$-nonadditive if its nonadditivity index satisfies $n_\mu(A) = 0$ for all $A \subseteq N$, $|A| > k$ and there exists at least one subset $A$ of $k$ elements such that $n_\mu(A) \neq 0$, where the nonadditivity index of a subset $A \subseteq N$ of $\mu$ is defined as
${n_\mu }(A) = \mu (A) - \frac{1}{{{2^{|A| - 1}} - 1}} \sum_{C \subset A} {\mu (C)}$.
\end{definition}

\begin{definition} \label{def. k nonmodular} \cite{wuBeliakovNonmodu,wuBeliakov-k-representative}
Let $k \in \{1,..., n\}$, a fuzzy measure $\mu$ on $N$ is said to be $k$-nonmodular if its nonmodularity index satisfies $d_\mu(A) = 0$ for all $A \subseteq N$, $|A| > k$ and there exists at least one subset $A$ of $k$ elements such that $d_\mu(A) \neq 0$, where the nonmodularity index of a subset $A \subseteq N$ of $\mu$ is defined as
${d_\mu }(A) =\mu (A) -\frac{1}{|A|}\sum_{\{i\} \subset A } [\mu(\{i\})+\mu(A\backslash\{i\})]$.
\end{definition}
	
\section{From additive measures to normal fuzzy measures}

\subsection{Linear extension and fuzzy measure generation}

The relationship between the linear extensions of all subsets in $N$ and the fuzzy measures on $N$ can be summarized as the following statement.
\begin{proposition}
  Any fuzzy measure can establish at least one linear extension, and any linear extension can generate numerous fuzzy measures.  
\end{proposition}

\begin{proof}
First, for a fuzzy measure with unique measure values for all subsets, we can obtain a unique linear extension of all subsets by arranging them in ascending order based on their measure values. If a fuzzy measure has the same values for different subsets, we can sort these subsets with equal measure values according to either the lexical order or binary order. 
Taking $N=\{1,2,3\}$ as an example, the lexical order of all subsets is given as follows.:
\begin{equation*}
		\emptyset, \{1\},\{2\},\{3\},\{1,2\},\{1,3\},\{2,3\},\{1,2,3\};
	\end{equation*}
and the binary order is given as:
	\begin{equation*}
	\emptyset, \{1\},\{2\},\{1,2\},\{3\},\{1,3\},\{2,3\},\{1,2,3\}.
    \end{equation*}
Second, if we have a linear extension of all subsets of $N$, we only need to assign $2^n$ ordered random values from the unit interval to this linear extension to generate a fuzzy measure. Therefore, numerous fuzzy measures can be generated with the same linear extension and these fuzzy measure are comonotonic ($\mu, \nu$ on $N$ are said to be comonotonic, if $\mu(A) \leq \mu(B)$ if and only if  $\nu(A) \leq \nu(B)$, $\forall A, B \subseteq N$.) .
\end{proof}

As special fuzzy measures with only $n$ coefficients, additive measures can be adopted as an efficient way to obtain linear extensions and further generate fuzzy measures. However, one main problem is the high repetition of linear extensions obtained by additive measures for small $n$, e.g, $n\leq 4$, and furthermore, large part of the linear extensions cannot be reached by additive measures.

\subsection{Reduce the repetition of linear extensions obtained by additive measures}

In order to mitigate the repetition of a group of existing linear extensions, we can randomly adjust the positions of some subsets without violating the monotonicity condition for each of these linear orders. For a subset $\emptyset \neq A \subset N$ within a linear extension, the following algorithm, \ref{alg-adj-one-subset}, can be employed to achieve a compatible adjustment of its position, where $\textbf{Pos}({A})$ represents its position index and $A= \arg\textbf{Pos}({A})$.

		\begin{algorithm} [!htb] 
		\caption{Adjust the position of a nonempty proper subset}
		\label{alg-adj-one-subset}
		$\textbf{Pos}_{L}= \max_{i \in A} \textbf{Pos} (A\backslash \{i\})$, \\
		$\textbf{Pos}_{U}= \min_{i \in N\backslash A} \textbf{Pos} (A \cup \{i\})$, \\	
  Reposition set $A$ randomly, if feasible, to a different location within $(\textbf{Pos}_{L}, \textbf{Pos}_{U})$.
	\end{algorithm} 

Indeed, the above operations can be equivalently performed through adjusting the associated additive measures. Suppose a linear extension is derived from an additive measure established by the nonzero measure values of $n$ singletons, i.e., the weight vector $(w_1, w_2,..., w_n)$ where $\sum_i w_i = 1$ and $\mu(A) = \sum_{i\in A} w_i$ for $A\subseteq N$. Then, the following alternative operations can be executed using Algorithm \ref{alg-adj-one-subset-by-measures}.

\begin{algorithm} [!htb] 
		\caption{Adjust the position of $A$ through measure values}
		\label{alg-adj-one-subset-by-measures}
		$\textbf{Int}_{L}= \max_{i \in A} \mu (A\backslash \{i\})$, \\
		$\textbf{Int}_{U}= \min_{i \in N\backslash A} \mu (A \cup \{i\})$, \\	
 Assign $\mu(A)$ a random value from the interval $[\textbf{Int}_{L}, \textbf{Int}_{U}]$, avoiding the range [$\mu(\arg (\textbf{Pos}(A)-1)), \mu(\arg (\textbf{Pos}(A)+1)]$, if feasible.
\end{algorithm} 

In algorithms \ref{alg-adj-one-subset} and \ref{alg-adj-one-subset-by-measures}, it is ensured that adjustments about any nonempty proper subset in $N$ do not break any monotonicity condition with respect to set inclusion.

\subsection{An efficient way for adjusting linear extension}

Another method to adjust linear extension is through random walk \cite{karzanov1991conductance,combarro2013On}, as shown in algorithm \ref{alg-random-walk} starting from the subset $A$ in the linear extension.
\begin{algorithm} [!htb] 
		\caption{Random walk for a linear extension}
		\label{alg-random-walk}
   \If {$|A| \geq |\arg (\textbf{Pos}(A)+1)|$ or $A \nsubseteq \arg (\textbf{Pos}(A)+1)$ }
    {Swap the positions of subsets $A, \arg (\textbf{Pos}(A)+1)$. }
\end{algorithm} 

Algorithm \ref{alg-random-walk} demonstrates that in a random walk, if $A$ is not a subset of $\arg (\textbf{Pos}(A)+1)$, their positions are exchanged. The random walk can be applied to a fuzzy measure $\mu$ by replacing the swap operation with setting $\mu(A)$ as a value between $\mu(\arg (\textbf{Pos}(A)+1))$ and $\mu(\arg (\textbf{Pos}(A)+2))$.

Actually, Algorithm \ref{alg-random-walk} can be altered to a different random walk direction by exchanging the values of $A$ and $\arg (\textbf{Pos}(A)-1)$, under the condition that {$|A| \leq |\arg (\textbf{Pos}(A)-1)|$ or $A\nsupseteq \arg (\textbf{Pos}(A)-1)$}. Consequently, the fuzzy measure value of $\mu(A)$ would be adjusted to lie within the range of $\mu(\arg (\textbf{Pos}(A)-1))$ and $\mu(\arg (\textbf{Pos}(A)-2))$.



Now let's assess the complexity of Algorithms \ref{alg-adj-one-subset} and \ref{alg-random-walk}. The former comprises set differences, minimization and maximization, position indexing, and insertion operations, leading to a complexity of at least $2o(|A|) + 2o(n) + 2o(1)$. The latter involves subset checks, number comparisons, and swap operations, resulting in a complexity ranging from $o(4)$ to $o(|A|) + o(4)$. Hence, one step of random walk proves to be more efficient.

Table \ref{tab-add-adjust-fuzzy} presents average results obtained from generating additive measures for universal sets with 3 to 6 elements over 10 iterations. The first through eighth columns correspond to the number of elements, the number of additive measures\footnotemark, the repetition ratio of linear extensions for the additive measures, the repetition ratio after applying algorithm \ref{alg-adj-one-subset}, and the repetition ratios for one through five steps of random walk, respectively. 
One can observe that one and two steps of random walk can yield comparable, and sometimes even better, results when compared to Algorithm \ref{alg-adj-one-subset}. Employing more steps of random walk can further enhance the performance. For the case of $n=6$, we can conclude that the repetition ratios are all zero for the extensive domain and numerous linear extensions.

\footnotetext{The total number of linear extensions for $n=1$ to $5$ are 1, 2, 48, 1680384, 14807804035657359360, and so on. This sequence is denoted as A046873 in OEIS, with additional values currently unknown \cite{combarro201950}. }

\begin{table}[!htb]\footnotesize
\caption{The repetition ratios of linear extensions.} 
		\label{tab-add-adjust-fuzzy}
\centering 
\begin{tabular}{rrrrrrrrr}
  \hline
 n& Num & Rep. &Alg. \ref{alg-adj-one-subset}  & RW-1 & RW-2 & RW-3 & RW-4 &RW-5 \\ 
  \hline
3&20  &0.6989 & 0.2675 & 0.3530 & 0.2680 & 0.2374 & 0.1910 & 0.1958 \\  
4&1000& 0.7579 & 0.1882 & 0.2228 & 0.0939 & 0.0329 & 0.0312 & 0.0268 \\ 
5 &10000 & 0.0214 & 0.0026 & 0.0020 & 0.0009 & 0.0003 & 0.0000 & 0.0000 \\ 
6 &100000& 0.0000 & 0.0000 & 0.0000 & 0.0000 & 0.0000 & 0.0000 & 0.0000 \\ 
   \hline
\end{tabular}
\end{table}

\section{Adjust additive measures into supermodular measures}

\subsection{Additive core and supermodular measure}


The core of a fuzzy measure $\mu$ on $N$ consists of the collection of additive measures $\nu$ on $N$ satisfying $\nu(A) \geq \mu(A)$ for all subsets $A \subseteq N$ \cite{shapley1953value,grabisch2016:setfunctionbook}. Supermodular measures, also known as convex games, are characterized by having a non-empty core, making them totally balanced \cite{shapley1971cores}. This fact inspires the following proposition, which outlines some strategies to adjust additive measures into supermodular measures.

\begin{proposition}
   For any additive measure $\nu$ on $N$, there exists at least one supermodular measure $\mu$ such that $\nu(A) \geq \mu(A)$ holds for all subsets $A \subseteq N$.
\end{proposition}

\begin{proof}
    The fact that an additive measure is still a supermodular measure ensures that this result holds for additive measures with some singleton's measure values being zeros. For an additive measure $\nu$ with $\nu(\{i\}) \neq 0, i\in N$, we can get a different supermodular measure through the following adjustment strategy:
    
    \textbf{(S1)}: $\mu(A)=\nu(A)-\eta$, $A\subset N, A\neq \emptyset$, $\eta \in (0, \min_{i\in N} \nu(\{i\}) )$.

\noindent
Taking account of the statement that a fuzzy measure $\mu$ on $N$ is supermodular if and only if \cite{grabisch2016:setfunctionbook} 
     \begin{equation} \label{eq-nonmo-mar}
		\Delta_i \mu(A) \geq \Delta_i \mu(A \backslash \{j\}),  j \in A,  i \notin A,  A \subset N, 
	\end{equation}
where $\Delta_i \mu(A)= \mu(A\cup \{i\})-\mu(A)$, we can observe that, for the $\mu$ adjusted from $\nu$, Eq. \eqref{eq-nonmo-mar} takes the direction of $>$ for $|A|=n-1,1$, and $=$ otherwise, indicating that it is a supermodular measure.
\end{proof}

\begin{corollary}
    For an additive measure $\nu$ with $\nu(\{i\}) \neq 0, \forall i\in N$, there exists a strictly supermodular measure $\mu$ such that $\nu(A) \geq \mu(A)$, $\forall A \subseteq N$.
\end{corollary}

\begin{proof}
    For the additive measure $\nu$, if $n=2$, we just use strategy (S1) to get a strictly supermodular measure. For $n \geq 3$, we can do the following adjustment strategy:
    
         \textbf{(S2)}: $\mu(A)=\nu(A)-\sum_{i=|A|}^{n-1}\eta_i, 1<|A|\leq n-1,$ where $0<\eta_1<\eta_2<\dots<\eta_{n-1}, \sum_{i=|A|}^{n-1}\eta_i < \min_{i\in N} \nu(\{i\}) $.

 \noindent      
We can have,  
\begin{equation*}
    \begin{split}
        \Delta_i \mu(A)&=\mu(A\cup \{i\})-\mu(A)
        \\&=\nu(A\cup \{i\})-\sum_{i=|A+1|}^{n-1}\eta_i-\nu(A)+\sum_{i=|A|}^{n-1}\eta_i
        \\&=\nu(\{i\})+ \eta_{|A|}
    \end{split}
\end{equation*}
Hence, $\Delta_i \mu(A) - \Delta_i \mu(A \backslash \{j\})= \nu(\{i\})+ \eta_{|A|}- \nu(\{i\})- \eta_{|A-1|}=\eta_{|A|}-eta-\eta_{|A-1|}>0$. According to Eq. \eqref{eq-nonmo-mar}, we can conclude that $\mu$ is a strictly supermodular measure.
 
\end{proof}






   




\subsection{Supermodularity adjusting through allowable range}

By employing adjustment strategies (S1) and (S2), we can transform an additive measure into a supermodular measure, while their linear extensions remain unchanged. To decrease the repetition ratio of linear extensions and enhance supermodularity, additional refinement of the measure values becomes necessary.

Actually, the Eq. \eqref{eq-nonmo-mar} can be rewritten as \cite{grabisch2016:setfunctionbook}
 \begin{equation} \label{eq-nonmo-quad}
\Delta_{ij} \mu(A) =\mu(A\cup\{i,j\})-\mu(A\cup\{i\})-\mu(A\cup\{j\})+\mu(A) \geq 0.
\end{equation}
Based on this equation, we can determine the permissible range for any nonempty proper subset $A$ to ensure supermodularity as:

\begin{algorithm} [!htb] 
\caption{Adjust $\mu(A)$ in a supermodular measure}
\label{alg-adj-measure-in-supermod}
$l_1=\max_{i,j\in A}(\mu(A\backslash \{i\})+\mu(A\backslash \{j\})-\mu(A\backslash \{i,j\}))$, \\
      $l_2=\max_{i,j\notin A}(\mu(A \cup \{i\})+\mu(A \cup \{j\})-\mu(A \cup \{i,j\}))$,\\
      $l_3=\min_{i\in A, j\notin A} (\mu(A \cup \{j\})-\mu(A \backslash \{i\} \cup \{j\})+ \mu(A \backslash \{i\}))$,\\	
 Assign $\mu(A)$ a random value in $[\max(l_1,l_2), l_3]$, avoiding the range [$\mu(\arg (\textbf{Pos}(A)-1)), \mu(\arg (\textbf{Pos}(A)+1)]$, if feasible.
\end{algorithm} 



By adjusting $\mu(A)$ within the allowable range $[\max(l_1,l_2), l_3]$, 
it will result in a new linear extension if $\mu(A)$ not in the range [$\mu(\arg (\textbf{Pos}(A)-1)), \mu(\arg (\textbf{Pos}(A)+1)]$. 




\subsection{Supermodularity adjusting through random walk}

Achieving a new linear extension through random walk for a supermodular measure requires additional checks on measure values beyond the set inclusion conditions listed in Algorithm \ref{alg-random-walk}.

\begin{algorithm} [!htb] 
		\caption{Random walk for supermodular measure}
		\label{alg-random-walk-supermodu}
   \If {$|A| \geq |\arg (\textbf{Pos}(A)+1)|$ or $A \nsubseteq \arg (\textbf{Pos}(A)+1)$ }
    { \If{$\mu(\arg (\textbf{Pos}(A)+1)) \leq l_3$ }
    {
    Set $\mu(A)$ as a value between $\mu(\arg (\textbf{Pos}(A)+1))$ and $\min(l_3, \mu(\arg (\textbf{Pos}(A)+2)))$.
    } }
\end{algorithm}

If the value of $\mu(A)$ is increased, then $\Delta_i \mu(A)$ will decrease and $\Delta_i \mu(A\backslash \{i\})$ will increase. To maintain supermodularity, it's necessary to verify the conditions $\Delta_i \mu(A)\geq \Delta_i \mu(A \backslash \{j\})$ for $i\notin A$ and $j \in A$, and $\Delta_i \mu(A\backslash \{i\}) \leq \Delta_i \mu(A\backslash \{i\} \cup \{j\} )$ for $ i\in A$ and $j\notin A$. However, these two conditions are essentially equivalent (both collapsing to the condition that $\mu(A)\leq l_3$ given in Algorithm \ref{alg-adj-measure-in-supermod}). The inclusion of $\min(l_3, \mu(\arg (\textbf{Pos}(A)+2)))$ in Algorithm \ref{alg-random-walk-supermodu} ensures that the position of $A$ is swapped only with $(\textbf{Pos}(A)+1))$ if it is feasible.

Conversely, if the value of $\mu(A)$ is decreased (this is also allowed), then $\Delta_i \mu(A)$ will increase and $\Delta_i \mu(A\backslash \{i\})$ will decrease. To maintain supermodularity after random walk in this direction, we need to check the conditions $\Delta_i \mu(A) \leq \Delta_i \mu(A \cup \{j\})$ for $i,j\notin A$, and $\Delta_i \mu(A\backslash \{i\}) \geq \Delta_i \mu(A\backslash \{i,j\} )$ for $ i,j\in A$, which is equivalent to check the conditions that $\mu(A)\geq l_2$ and $\mu(A)\geq l_1$, as given in Algorithm \ref{alg-adj-measure-in-supermod}. Then if feasible, the $\mu(A)$ can be set as a value between $\mu(\arg (\textbf{Pos}(A)-1))$ and $\max (l_1,l_2, \mu(\arg (\textbf{Pos}(A)-2)))$.

\subsection{Generate antibuoyant measure from uniform measure}

The antibuoyant measure is a special type of supermodular measure \cite{beliakov2022choquet}, which can be defined through marginal contributions. 
A fuzzy measure is antibuoyant if \cite{beliakov2021choquet}
 \begin{equation*}
 \Delta_i( A \cup \{j\}) \geq \Delta_j( A ), \forall A\subseteq N \backslash \{i,j\}, \forall i,j \in N.
    \end{equation*}

The core of an antibuoyant measure is a uniform (additive and symmetric) measure \cite{beliakov2022choquet}, specifically an additive measure $\nu$ on $N$ with $\nu(i)=\frac{1}{n}$. Therefore, starting from the uniform measure, we can apply the adjustment strategies \textbf{(S1)} and \textbf{(S2)} to obtain an antibuoyant measure and a strictly antibuoyant measure, respectively.

The allowable range of $\mu(A)$ that ensures still an antibuoyant measure should be:
\begin{equation}
    \mu(A) \in [\max(l_4, l_5), l_6]
\end{equation}
where 
$l_4=\max_{i,j \in A} (2 \mu(A \backslash \{i\})-\mu(A \backslash \{i,j\}, 2 \mu(A \backslash \{j\})-\mu(A \backslash \{i,j\} )$, $l_5=\max_{i,j \notin A} (2 \mu(A \cup \{i\})+\mu(A \cup \{i,j\}, 2 \mu(A \cup \{j\})+\mu(A \cup \{i,j\}))$, $l_6=\min_{i \notin A, j\in A} \frac{1}{2}(\mu(A \cup \{i\}+\mu(A \backslash \{j\}))$.

The random walk of $\mu(A)$ for an antibuoyant measure can be given as

\begin{algorithm} [!htb] 
		\caption{Random walk for antibuoyant measure}
		\label{alg-random-walk-antibuo}
   \If {$|A| \geq |\arg (\textbf{Pos}(A)+1)|$ or $A \nsubseteq \arg (\textbf{Pos}(A)+1)$ }
    {
     \If{$\mu(\arg (\textbf{Pos}(A)+1)) \leq l_6$ }
    {
    Set $\mu(A)$ as a value between $\mu(\arg (\textbf{Pos}(A)+1))$ and $\min(l_6, \mu(\arg (\textbf{Pos}(A)+2)))$.
    } 
    }
\end{algorithm} 

On the contrary, if we want to decrease the value of $\mu(A)$ to get a new linear extension for the antibuoyant measure, we need to check whether $\mu(\arg (\textbf{Pos}(A)-1)) \geq l_4$ and $\mu(\arg (\textbf{Pos}(A)-1)) \geq l_5$, then if feasible, set $\mu(A)$ as a value between $\mu(\arg (\textbf{Pos}(A)-1))$ and $\min(l_4, l_5,  \mu(\arg (\textbf{Pos}(A)-2)))$.

\section{Adjust additive measures into superadditive measures}

Supermodularity trivially implies superadditivity. Hence for an additive measure $\nu$ on $N$, there must exist a superadditive measure $\mu$ such that $\mu (A)\leq \nu(A), A \subseteq N$.

\begin{proposition}
   For any additive measure $\nu$ on $N$, n>3, with $ i_0 \in N, \nu(\{i_0\}) \neq 0$,  there exists at least one superadditive, but not necessarily supermodular, measure $\mu$ such that $\nu(A) \geq \mu(A)$ holds for all subsets $A \subseteq N$.
\end{proposition}

\begin{proof}
 For a subset $B \subset N, |B|>2, i_0 \in B$, if we use the following adjustment strategy: 
\begin{equation*}
     \textbf{(S3): } \mu(A)= 
     \begin{cases}
      \nu(A)-\eta &\text{ if } i_0 \in A\subseteq B; \\
      \nu(A) &\text{ otherwise }. 
 \end{cases}
\end{equation*}
where $\eta \in (0, \nu(\{i_0\})$. We can have $\mu$ as a superadditive measure satisfying definition \ref{def.superadd}, but the Eq. \eqref{eq-nonmo-mar} will not necessarily hold, and therefore, it might not be a supermodular measure.

\end{proof}

\begin{corollary}
    For an additive measure $\nu$ on $N$, n>3, with $\nu(\{i\}) \neq 0,  \forall i\in N$, there exists a strictly superadditive, but not necessarily supermodular, measure $\mu$ such that $\nu(A) \geq \mu(A)$, $\forall A \subseteq N$.
\end{corollary}

\begin{proof}
    By using the following adjustment strategy:
    
    \textbf{(S4):} For each $A \subseteq N, |A|=n-1$, $\mu(B)=\nu(B)-\eta_A, \eta_A \in (0, \frac{\min_{i \in A} \nu(\{i\})}{n})$,

    \noindent we can have a strictly superadditive, but not necessarily supermodular, measure.
\end{proof}

Without violating the superadditivity of a fuzzy measure $\mu$, we can further adjust the measure value of a  subset $A, \emptyset \neq A\subset N,$ within its allowable range as:
\begin{equation} \label{eq-nonadd-allowable-range}
    \mu(A) \in [l_7, l_8]
\end{equation}
where $l_7=\max_{\emptyset \neq B\subset A} (\mu(B)+\mu(A\backslash B))$, $l_8=\min_{A\subset C \subset N} (\mu(C)-\mu(C\backslash A))$.
Similarly, if the adjusted value of $\mu(A)$ is not in the interval $[\mu(\arg (\textbf{Pos}(A)-1)), \mu(\arg (\textbf{Pos}(A)+1)]$, a new linear extension for  superadditive measure can be obtained. 

Alternatively, the following random walk can be employed for a superadditive measure, see algorithm \ref{alg-random-walk-superadd}.

\begin{algorithm} [!htb] 
		\caption{Random walk for superadditive measure}
		\label{alg-random-walk-superadd}
   \If {$|A| \geq |\arg (\textbf{Pos}(A)+1)|$ or $A \nsubseteq \arg (\textbf{Pos}(A)+1)$ }
    {
     \If{$\mu(\arg (\textbf{Pos}(A)+1)) \leq l_8$ }
    {
    Set $\mu(A)$ as a value between $\mu(\arg (\textbf{Pos}(A)+1))$ and $\min(l_8, \mu(\arg (\textbf{Pos}(A)+2)))$.
    } 
    }
\end{algorithm} 

On the contrary, if we want to decrease the value of $\mu(A)$ to get a new linear extension, we need to check whether $\mu(\arg (\textbf{Pos}(A)-1)) \geq l_7$, then if feasible, set $\mu(A)$ as a value between $\mu(\arg (\textbf{Pos}(A)-1))$ and $\min(l_7, \mu(\arg (\textbf{Pos}(A)-2)))$.





\section{Generation of special measures from additive measures}
\subsection{ Vectors based generations for $p$-symmetric measures}

The allowable range and random walk of regular fuzzy measures can be smoothly adapted to $p$-symmetric measures \cite{combarro201950} by representing any subset $S$ as a $p$-dimensional vector $\mathbf{b}_S=(b_1,..., b_p)$ as defined in \ref{def.p.symmetric}. 
The subset inclusion relation can be transformed into a partial order of $p$-dimensional vectors: $\mathbf{a} \leq \mathbf{b}$ if $a_i \leq b_i$ for $i \in (1, \ldots, p)$. The relationship between sets $A$ and $B=A\backslash \{i\}$ can be represented as $\sum(\mathbf{b}_{A}-\mathbf{b}_{B})=1$. 
For the initial additive measures for $p$-symmetric measures, the singletons in the same partition of the basis should have the same weight, i.e., $\mathbf{b}_{\{i\}}=\mathbf{b}_{\{j\}}$ implies $\nu(\{i\})=\nu(\{j\})$.
As a result, the algorithms and adjustment strategies for normal fuzzy measures, supermodular, superadditive, and buoyant measures are applicable to $p$-symmetric measures. Actually, any normal fuzzy measure can be seen as the $n$-symmetric measure with the basis of $N$, with all subsets being represented as 0-1 valued $n$-dimensional vectors.

\subsection{Adjustments focusing on lower $k$-order subsets}

When considering $k$-tolerant measures, $k$-interactive measures, and $k$-maxitive measures, our primary focus lies in addressing the random walk or measure value adjustments for subsets with orders lower than $k$.
Suppose we already have a fuzzy measure $\nu(A)$ on $N$, no matter it is normal, additive, supermodular, antibuoyant or superadditive fuzzy measure, we can get
\begin{itemize}
\item  a $k$-tolerant measure $\mu$ by setting
\begin{equation}
    \mu(A)= \begin{cases}
         \mu(A)=\nu(A)  & |A|\leq k;\\
         \mu(A)=1     &|A|\geq k+1.
    \end{cases}
\end{equation}
\item a $k$-interactive measure $\mu$ by setting
\begin{equation}
    \mu(A)= \begin{cases}
         \mu(A)=\frac{K \nu(A)}{\max_{|B|=k} \nu(B)}  & |A|\leq k;\\
         \mu(A)=K+\frac{|A|-k-1}{n-k-1}(1-K)     &|A|\geq k+1.
    \end{cases}
\end{equation}
\item a $k$-maxitive measure $\mu$ by setting
\begin{equation}
    \mu(A)= \begin{cases}
         \mu(A)=\frac{\nu(A)}{\max_{|B|=k} \nu(B)}  & |A|\leq k;\\
         \mu(A)=\max_{|B|=|A|-1} \mu(B)    &|A|\geq k+1.
    \end{cases}
\end{equation}
\end{itemize}

We can verify that all the operations mentioned above consistently maintain the monotonicity and normalization conditions, while also satisfying the specific requirements of the three special types of measures.

\subsection{More complex computations for higher order subsets}

For $k$-additive measures, $k$-nonadditive measures, and $k$-nonmodular measures, maintaining the zero value for $k+1$ and higher interaction indices magnifies the complexity of adjustment strategies and random walks for the lower $k$-order subsets.

\subsubsection{$k$-additive measure} 

To obtain a $k$-order additive measure $\mu$ from an additive measure $\nu$, we need to adjust the lower $k$-order subsets according to the requirement of zero values for $k+1$ and higher order M\"obius representations.

For an additive measure $\nu$ on $N$, we can adopt the following allowable range for $\nu(A), |A|\leq k,$  to keep monotonicity conditions as well as zero values for $k+1$ and higher order M\"obius representations:
\begin{equation} \label{eq.allow-range-k-additive}
     [\nu(A)-\min(a_1,a_2,a_3), \nu(A)+\min(a_4,a_5,a_6)]
\end{equation}
where $a_1=\nu(A)-\max_{i \in A} \nu (A\backslash \{i\}),$
\begin{equation*}
   a_2= \begin{cases}
       \min_{j \in B, |B|= k+1, A\subset B }(\nu(B)-  \nu(B\backslash \{j\})) & \text{if $k+1-|A|$ is even}; \\
\min_{j \in N\backslash B, |B|= k+1, A\subset B } (\nu(B\cup \{j\})-\nu(B))/2 & \text{if $k+1-|A|$ is odd}, \\
\end{cases}
\end{equation*}
\begin{equation*}
   a_3= \begin{cases}
\min_{|B|> k+1, A\subset B}((\nu(B)-\max_{i \in A }  \nu(B\backslash \{i\}))/2,&\\
\qquad\qquad\qquad\quad \nu(B)- \max_{j \in B\backslash A } \nu(B\backslash \{j\}))  & \text{if $|B|-|A|$ is even}; \\
\min_{j \in N\backslash B, |B|> k+1, A\subset B } (\nu(B\cup \{j\})-\nu(B))/2 & \text{if $|B|-|A|$ is odd},\\
\end{cases}
\end{equation*}


$a_4= \min_{i \in N\backslash A} \nu (A \cup \{i\})-\nu(A)$,	
\begin{equation*}
   a_5= \begin{cases}
       \min_{j \in N\backslash B, |B|= k+1, A\subset B } (\nu(B\cup \{j\})-\nu(B))/2 & \text{if $k+1-|A|$ is even}; \\
\min_{j \in B, |B|= k+1, A\subset B }(\nu(B)-  \nu(B\backslash \{j\})) & \text{if $k+1-|A|$ is odd}, \\
\end{cases}
\end{equation*}
\begin{equation*}
   a_6= \begin{cases}
   \min_{j \in N\backslash B, |B|> k+1, A\subset B } (\nu(B\cup \{j\})-\nu(B))/2 & \text{if $|B|-|A|$ is even};\\
\min_{|B|> k+1, A\subset B}((\nu(B)-\max_{i \in A }  \nu(B\backslash \{i\}))/2,&\\
\qquad\qquad\qquad\quad \nu(B)- \max_{j \in B\backslash A } \nu(B\backslash \{j\}))  & \text{if $|B|-|A|$ is odd}. \\
\end{cases}
\end{equation*}
Therefore, the adjustment strategy for obtaining a $k$-additive measure $\mu$ from an additive measure $\nu$ can be as follows:

\textbf{(S5):} choose a subset $A$, $0<|A|\leq k$, assign $\mu(A)$ a value within $[\nu(A)-\min(a_1,a_2,a_3), \nu(A)+\min(a_4,a_5,a_6)]$ that is not equal to $\nu(A)$ if feasible, then to ensure the zero values for $k+1$ and higher order M\"obius representations, then we need to set
\begin{equation}
    \mu(B)=\begin{cases} \label{eq.adjust.k.additive}
        \nu(B)     &\text{if } B\neq A, |B|\leq k, \\&\text{or } A \nsubseteq B, |B|> k;\\
        \nu(B)+(-1)^{|B|-|A|}(\mu(A)-\nu(A))  &\text{if }  A\subset B, |B|>k.\\
    \end{cases}
\end{equation}
Finally, normalize measure values by setting $\mu(B)=\mu(B)/\mu(N)$ for all $B \subseteq N$.

The random walk for the lower $k$-order subsets can be given in Algorithm \ref{alg-random-walk-k-additive}:
\begin{algorithm} [!htb] 
\caption{Random walk for $k$-additive measure}
\label{alg-random-walk-k-additive}
   \If {$|A| \geq |\arg (\textbf{Pos}(A)+1)|$ or $A \nsubseteq \arg (\textbf{Pos}(A)+1)$ }
    { \If{$\mu(\arg (\textbf{Pos}(A)+1)) \leq \mu(A)+\min(a_4,a_5,a_6)$ }
    {
    Set $\mu(A)$ as a value between $\mu(\arg (\textbf{Pos}(A)+1))$ and $\min(\mu(A)+\min(a_4,a_5,a_6), \mu(\arg (\textbf{Pos}(A)+2)))$,\\
    Adjust the measure vales according to strategy \textbf{(S5)}.
    } }
\end{algorithm} 

Similarly, if we intend to perform a random walk by decreasing the value of $\mu(A)$, a feasibility check should be conducted to ensure that $\mu(\arg (\textbf{Pos}(A)-1)) \geq \mu(A) - \min(a_1, a_2, a_3)$. If this condition is met, we can set $\mu(A)$ to a value between $\mu(\arg (\textbf{Pos}(A)-1))$ and $\min(\nu(A) - \min(a_1, a_2, a_3), \mu(\arg (\textbf{Pos}(A)-2)))$, while still needing to adjust the measure values according to strategy \textbf{(S5)}.

\subsubsection{$k$-nonadditive measure}

Aiming to obtain a $k$-nonadditive measure from an additive measure $\nu$ on $N$, we can employ the following allowable range for $\nu(A)$, where $|A|\leq k$, to maintain the monotonicity conditions as well as ensuring zero values for $k+1$ and higher order nonadditivity indices:
\begin{equation} \label{eq.allow-range-k-nonadditive}
     [\nu(A)-\min(b_1,b_2), \nu(A)+\min(b_3,b_4)]
\end{equation}
where $b_1=\nu(A)-\max_{i \in A} \nu (A\backslash \{i\}),$

$b_2= \min_{i \in A \subset B, B|>k} (2^{|B|-1}-1)(\mu(B)-\mu(B\backslash \{i\}))$,

$b_3= \min_{i \in N\backslash A} \nu (A \cup \{i\})-\nu(A)$,	

$b_4=\min_{i \in N\backslash B, A \subset B} \frac {(2^{|B|-1}-1) (2^{|B|}-1)} {2^{|B|-1} }(\mu(B \cup \{i\})-\mu(B))$.


Hence, the adjustment strategy to obtain a $k$-nonadditive measure $\mu$ from an additive measure $\nu$ can be given as follows:

\textbf{(S6):} choose a subset $A$, $0<|A|\leq k$, assign $\mu(A)$ a value within $[\nu(A)-\min(b_1,b_2), \nu(A)+\min(b_3,b_4)]$ that is not equal to $\nu(A)$ if possible, then to ensure the zero values for $k+1$ and higher order nonadditivity indices, then we set
\begin{equation}
    \mu(B)=\begin{cases} \label{eq.adjust.k.nonadditive}
        \nu(B)     &\text{if } B\neq A, |B|\leq k, \\&\text{or } A \nsubseteq B, |B|> k;\\
        \nu(B)+\frac{1}{2^{|B|-1}-1}(\mu(A)-\nu(A))  &\text{if }  A\subset B, |B|>k.\\
    \end{cases}
\end{equation}
Finally, normalize measure values by setting $\mu(B)=\mu(B)/\mu(N)$ for all $B \subseteq N$.

The random walk for the lower $k$-order subsets of $k$-nonadditive measure can be given in Algorithm \ref{alg-random-walk-k-nonadditive}:
\begin{algorithm} [!htb] 
\caption{Random walk for $k$-nonadditive measure}
\label{alg-random-walk-k-nonadditive}
   \If {$|A| \geq |\arg (\textbf{Pos}(A)+1)|$ or $A \nsubseteq \arg (\textbf{Pos}(A)+1)$ }
    { \If{$\mu(\arg (\textbf{Pos}(A)+1)) \leq \mu(A)+\min(b_3,b_4)$ }
    {
    Set $\mu(A)$ as a value between $\mu(\arg (\textbf{Pos}(A)+1))$ and $\min(\mu(A)+\min(b_3,b_4), \mu(\arg (\textbf{Pos}(A)+2)))$,\\
    Adjust the measure vales according to strategy \textbf{(S5)}.
    } }
\end{algorithm} 

Similarly, for a $k$-nonadditive measure, if the goal is to perform a random walk by decreasing the value of $\mu(A)$, a feasibility check should be conducted to ensure that $\mu(\arg (\textbf{Pos}(A)-1)) \geq \mu(A) - \min(b_1, b_2)$. If this condition is satisfied, $\mu(A)$ can be set to a value between $\mu(\arg (\textbf{Pos}(A)-1))$ and $\min(\nu(A) - \min(b_1, b_2), \mu(\arg (\textbf{Pos}(A)-2)))$, while still needing to adjust the measure values according to strategy \textbf{(S6)}.


\subsubsection{k-nonmodular measure}

Adjusting an additive measure $\nu$ on $N$ into a $k$-nonmodular measure, we can employ the following allowable range for $\nu(A)$, $|A|\leq k$, to maintain the monotonicity conditions as well as ensuring zero values for $k+1$ and higher order nonmodularity indices:
\begin{equation} \label{eq.allow-range-k-nonmodular}
     [\nu(A)-\min(c_1,c_2), \nu(A)+\min(c_3,c_4)]
\end{equation}
where $c_1=\nu(A)-\max_{i \in A} \nu (A\backslash \{i\}),$
\begin{equation*}
    c_2=\begin{cases}
        1                & \textbf{ if } |A|>1;\\
        \min_{i\in B, |B|=k+1, A \subset B} |B|(\nu(B)- \nu(B\backslash \{i\}))                & \textbf{ if } |A|=1.\\
    \end{cases}
\end{equation*}

$c_3= \min_{i \in N\backslash A} \nu (A \cup \{i\})-\nu(A)$,	
\begin{equation*}
   c_4= \begin{cases}
        \min_{i \in N\backslash B, A \subset B} (|B|)(|B|+1)(\mu(B \cup \{i\}) -\mu(B)) & \textbf{ if } |A|=1;\\
         \min_{i \in N\backslash A, B=A \cup \{i\}} \frac {|B|} {|B|-1} (\mu(B)-\mu(A)) & \textbf{ if } |A|=k.\\
    \end{cases}
\end{equation*}

Hence, the adjustment strategy to obtain a $k$-nonmodular measure $\mu$ from an additive measure $\nu$ can be given as follows:

\textbf{(S7):} choose a subset $A$, $0<|A|\leq k$, assign $\mu(A)$ a value within $[\nu(A)-\min(c_1,c_2), \nu(A)+\min(c_3,c_4)]$ that is not equal to $\nu(A)$ if possible, then to ensure the zero values for $k+1$ and higher order nonadditivity indices, then we set
\begin{equation}
    \mu(B)=\begin{cases} \label{eq.adjust.k.nonmodular}
        \nu(B)     &\text{if } B\neq A, |B|\leq k, \\&\text{or } A \nsubseteq B, |B|> k;\\
        \nu(B)+\frac{1}{|B|}(\mu(A)-\nu(A))  &\text{if }  |A|=1, A\subset B, |B|>k.\\
        \nu(B)+\frac{1}{|B|}(\mu(A)-\nu(A))  &\text{if }  |A|=k, A\subset B, |B|=k+1.\\      
    \end{cases}
\end{equation}
Finally, normalize measure values by setting $\mu(B)=\mu(B)/\mu(N)$ for all $B \subseteq N$.

The random walk for the lower $k$-order subsets of $k$-nonmodular measure can be given in Algorithm \ref{alg-random-walk-k-nonmodular}:
\begin{algorithm} [!htb] 
\caption{Random walk for $k$-nonmodular measure}
\label{alg-random-walk-k-nonmodular}
   \If {$|A| \geq |\arg (\textbf{Pos}(A)+1)|$ or $A \nsubseteq \arg (\textbf{Pos}(A)+1)$ }
    { \If{$\mu(\arg (\textbf{Pos}(A)+1)) \leq \mu(A)+\min(c_3,c_4)$ }
    {
    Set $\mu(A)$ as a value between $\mu(\arg (\textbf{Pos}(A)+1))$ and $\min(\mu(A)+\min(c_3,c_4), \mu(\arg (\textbf{Pos}(A)+2)))$,\\
    Adjust the measure vales according to strategy \textbf{(S5)}.
    } }
\end{algorithm} 

Similarly, when aiming to decrease $\mu(A)$ through a random walk for a $k$-nonmodular measure, it's essential to first check if $\mu(\arg (\textbf{Pos}(A)-1)) \geq \mu(A) - \min(c_1, c_2)$. If this holds true, you can adjust $\mu(A)$ to a value between $\mu(\arg (\textbf{Pos}(A)-1))$ and $\min(\nu(A) - \min(c_1, c_2), \mu(\arg (\textbf{Pos}(A)-2)))$, while further execute the additional adjustments by using strategy \textbf{(S7)}.

\section{Conclusions}
In this study, we have developed some approaches for generating a variety of fuzzy measures, which use additive measures as initial solutions and employ techniques such as allowable range adjustments, random walks, and adjustment strategies to acquire a comprehensive range of measures. These encompass normal fuzzy measures, supermodular measures, antibuoyant measures, superadditive measures, $k$-tolerant measures, $k$-interactive measures, $k$-maxitive measures, $k$-additive measures, $k$-nonadditive measures, and $k$-nonmodular measures.

By extending this framework with the concept of dual measures, we can derive additional measures, including submodular, antibuoyant, subadditive, $k$-intolerant, $k$-minitive, upper $k$-additive, upper $k$-nonadditive, and upper-nonmodular measures. Importantly, the adaptability of this methodology enables the integration of super/submodular or super/subadditive measures with $k$-order measures, facilitated by refined adjustment strategies and novel random walk methods.

Furthermore, we find potential for direct measure value swaps among neighboring subsets within the random walk algorithm, warranting further exploration. In future investigations, we intend to delve into the application of allowable range adjustments, random walks, and adjustment strategies to address more intricate constraints. These comprehensive approaches not only enhances our understanding of the interplay between additive and fuzzy measures but also provides novel insights and pragmatic applications across diverse domains.

\section*{Acknowledgements}
	The work was supported by the Australian Research Council Discovery project DP210100227.

\bibliographystyle{abbrv}
\bibliography{sample}

\begin{thebibliography}{10}

\bibitem{beliakov2021choquet}
G.~Beliakov and S.~James.
\newblock Choquet integral optimisation with constraints and the buoyancy
  property for fuzzy measures.
\newblock {\em Information Sciences}, 578:22--36, 2021.

\bibitem{beliakov2022choquet}
G.~Beliakov and S.~James.
\newblock Choquet integral-based measures of economic welfare and species
  diversity.
\newblock {\em International Journal of Intelligent Systems}, 37(4):2849--2867,
  2022.

\bibitem{beliakov2011learning}
G.~Beliakov, S.~James, and G.~Li.
\newblock Learning {C}hoquet-integral-based metrics for semisupervised
  clustering.
\newblock {\em IEEE Transactions on Fuzzy Systems}, 19(3):562--574, 2011.

\bibitem{beliakov2019discrete_book}
G.~Beliakov, S.~James, and J.-Z. Wu.
\newblock {\em Discrete Fuzzy Measures: Computational Aspects}.
\newblock Springer, Cham, Switzerland, 2019.

\bibitem{Beliakov2007_book}
G.~Beliakov, A.~Pradera, and T.~Calvo.
\newblock {\em Aggregation Functions: A Guide for Practitioners}.
\newblock Springer, Berlin, Heidelberg, 2007.

\bibitem{beliakovWuLearning}
G.~Beliakov and J.-Z. Wu.
\newblock Learning fuzzy measures from data: simplifications and optimisation
  strategies.
\newblock {\em Information Sciences}, 494:100--113, 2019.

\bibitem{beliakov2020kmaxtive}
G.~Beliakov and J.-Z. Wu.
\newblock Learning k-maxitive fuzzy measures from data by mixed integer
  programming.
\newblock {\em Fuzzy Sets and Systems}, 2020.

\bibitem{beliakov-wu-towards}
G.~Beliakov, J.-Z. Wu, and D.~Divakov.
\newblock Towards sophisticated decision models: Nonadditive robust ordinal
  regression for preference modeling.
\newblock {\em Knowledge-Based Systems}, 190:105351, 2020.

\bibitem{choquet1954}
G.~Choquet.
\newblock Theory of capacities.
\newblock {\em Annales de l'institut Fourier}, 5:131--295, 1954.

\bibitem{combarro201950}
E.~F. Combarro, J.~H. de~Saracho, and I.~Díaz.
\newblock Minimals plus: An improved algorithm for the random generation of
  linear extensions of partially ordered sets.
\newblock {\em Information Sciences}, 501:50 -- 67, 2019.

\bibitem{combarro2013On}
E.~F. Combarro, I.~Díaz, and P.~Miranda.
\newblock On random generation of fuzzy measures.
\newblock {\em Fuzzy Sets \& Systems}, 228(4):64--77, 2013.

\bibitem{denneberg1994non}
D.~Denneberg.
\newblock {\em Non-additive Measure and Integral}.
\newblock Springer Science \& Business Media, Dordrecht, 1994.

\bibitem{grabisch1995fuzzy}
M.~Grabisch.
\newblock Fuzzy integral in multicriteria decision making.
\newblock {\em Fuzzy Sets and Systems}, 69(3):279--298, 1995.

\bibitem{grabisch1997k}
M.~Grabisch.
\newblock k-order additive discrete fuzzy measures and their representation.
\newblock {\em Fuzzy Sets and Systems}, 92(2):167--189, 1997.

\bibitem{grabisch2016:setfunctionbook}
M.~Grabisch.
\newblock {\em Set Functions, Games and Capacities in Decision Making}.
\newblock Springer, Berlin, New York, 2016.

\bibitem{grabisch2008review}
M.~Grabisch, I.~Kojadinovic, and P.~Meyer.
\newblock A review of methods for capacity identification in {C}hoquet integral
  based multi-attribute utility theory: Applications of the {K}appalab {R}
  package.
\newblock {\em European Journal of Operations Research}, 186(2):766--785, 2008.

\bibitem{karzanov1991conductance}
A.~Karzanov and L.~Khachiyan.
\newblock On the conductance of order markov chains.
\newblock {\em Order}, 8:7--15, 1991.

\bibitem{marichal2004tolerant}
J.-L. Marichal.
\newblock Tolerant or intolerant character of interacting criteria in
  aggregation by the {C}hoquet integral.
\newblock {\em European Journal of Operational Research}, 155(3):771--791,
  2004.

\bibitem{marichal2007k}
J.-L. Marichal.
\newblock k-intolerant capacities and {C}hoquet integrals.
\newblock {\em European Journal of Operational Research}, 177(3):1453--1468,
  2007.

\bibitem{mesiar1999generalizations}
R.~Mesiar.
\newblock Generalizations of k-order additive discrete fuzzy measures.
\newblock {\em Fuzzy Sets and Systems}, 102(3):423--428, 1999.

\bibitem{mesiar2018k}
R.~Mesiar and A.~Koles{\'a}rov{\'a}.
\newblock k-maxitive aggregation functions.
\newblock {\em Fuzzy Sets and Systems}, 346:127--137, 2018.

\bibitem{miranda2002p}
P.~Miranda, M.~Grabisch, and P.~Gil.
\newblock p-symmetric fuzzy measures.
\newblock {\em International Journal of Uncertainty, Fuzziness and
  Knowledge-Based Systems}, 10(supp01):105--123, 2002.

\bibitem{shapley1953value}
L.~S. Shapley.
\newblock A value for n-person games.
\newblock {\em Contributions to the Theory of Games}, 2(28):307--317, 1953.

\bibitem{shapley1971cores}
L.~S. Shapley.
\newblock Cores of convex games.
\newblock {\em International journal of game theory}, 1:11--26, 1971.

\bibitem{sugeno1974theory}
M.~Sugeno.
\newblock {\em Theory of Fuzzy Integrals and Its Applications}.
\newblock PhD thesis, Tokyo Institute of Technology, 1974.

\bibitem{wuBeliakovNonadd}
J.-Z. Wu and G.~Beliakov.
\newblock Nonadditivity index and capacity identification method in the context
  of multicriteria decision making.
\newblock {\em Information Sciences}, 467:398--406, 2018.

\bibitem{wuBeliakovkminitive}
J.-Z. Wu and G.~Beliakov.
\newblock k-minitive capacities and k-minitive aggregation functions.
\newblock {\em Journal of Intelligent and Fuzzy Systems}, 37(2):2797–2808,
  2019.

\bibitem{wuBeliakovNonmodu}
J.-Z. Wu and G.~Beliakov.
\newblock Nonmodularity index for capacity identifying with multiple criteria
  preference information.
\newblock {\em Information Sciences}, 492:164--180, 2019.

\bibitem{wuBeliakov-k-representative}
J.-Z. Wu and G.~Beliakov.
\newblock k-order representative capacity.
\newblock {\em Journal of Intelligent and Fuzzy Systems}, 38(3):3105--3115,
  2020.

\end{thebibliography}

\end{document}